\documentclass[11pt]{article} 

\usepackage{fullpage}
\usepackage{amsmath, amsthm, amssymb, mathrsfs, bm, mathtools}
\usepackage{ragged2e, enumerate, graphicx, xcolor}
\usepackage{subfiles}
\usepackage{tabularray}

\usepackage{pgf, tikz, float, subcaption}
\usetikzlibrary{graphs, graphs.standard}

\makeatletter
\def\@seccntDot{.}
\def\@seccntformat#1{\csname the#1\endcsname\@seccntDot\hskip 0.5em}
\renewcommand\section{\@startsection{section}{1}{\z@}%
{18\p@ \@plus 6\p@ \@minus 3\p@}%
{9\p@ \@plus 6\p@ \@minus 3\p@}%
{\large\bfseries\boldmath}}
\renewcommand\subsection{\@startsection{subsection}{2}{\z@}%
{12\p@ \@plus 6\p@ \@minus 3\p@}%
{3\p@ \@plus 6\p@ \@minus 3\p@}%
{\bfseries\boldmath}}
\renewcommand\subsubsection{\@startsection{subsubsection}{3}{\z@}%
{12\p@ \@plus 6\p@ \@minus 3\p@}%
{\p@}%
{\bfseries\boldmath}}
\makeatother

\usepackage{microtype}

\theoremstyle{plain}
\newtheorem{theorem}{Theorem}[section]
\newtheorem{lemma}[theorem]{Lemma}
\newtheorem{proposition}[theorem]{Proposition}
\newtheorem{conjecture}[theorem]{Conjecture}

\theoremstyle{definition}
\newtheorem{claim}{Claim}[section]
\newtheorem*{remark}{Remark}

\usepackage[pagebackref]{hyperref}
\hypersetup{
 colorlinks=true,
 urlcolor=purple,
 linkcolor=purple,
 citecolor=purple,
}

\newcommand{\affl}[3]{\noindent #1, Email: {\tt #2}\\ \textsc{#3}\\[1.5pt]}
\numberwithin{equation}{section}
\allowdisplaybreaks

\newcommand{\dif}{\mathop{}\!\mathrm{d}}
\DeclareMathOperator{\tr}{tr}
\DeclareMathOperator{\diag}{diag}

\title{\textbf{Maximum spectral sum of graphs}}
\author{Hitesh Kumar, \ Lele Liu, \ Hermie Monterde, \ Shivaramakrishna Pragada, \ Michael Tait}
\date{}

\begin{document}
\maketitle

\begin{abstract}
For a graph $G$ of order $n$, the spectral sum of $G$ is defined to be the sum $\lambda_1(G) + \lambda_2(G)$, 
where $\lambda_1(G)$ (resp. $\lambda_2(G)$) is the largest (resp. second largest) adjacency eigenvalue of $G$. 
Ebrahimi, Mohar, Nikiforov and Ahmady (2008) conjectured that the spectral sum
\[ 
\lambda_1(G) + \lambda_2(G)\le \frac{8}{7}n
\] 
for any graph $G$. We prove this conjecture by combining tools from the theory of graph limits, 
convex geometry, exterior algebra and convex optimization. The techniques developed are of independent interest.
\end{abstract}



\section{Introduction}

Consider a simple undirected graph $G = (V,E)$ with vertex set $V = \{1,2,\ldots,n\}$. 
The \emph{adjacency matrix} of $G$ is defined to be the square matrix $A(G)=[a_{ij}]$ of order $n$, where $a_{ij}=1$ if 
$i$ and $j$ are adjacent, and $a_{ij}=0$ otherwise. We order the eigenvalues of $A(G)$ as 
$\lambda_1(G)\geq\lambda_2(G)\geq\cdots\geq\lambda_n(G)$. The first eigenvalue $\lambda_1(G)$ is called 
the \emph{spectral radius} of $G$. The sum of the first two largest eigenvalues $\lambda_1(G) + \lambda_2(G)$ 
is called the \emph{spectral sum} of $G$. For $n\in \mathbb{N}$, let $\mathcal{G}(n)$ 
denote the set of simple undirected graphs on $n$ vertices. We denote the \emph{complement} of $G$ by $\overline{G}$.

Extremization of graph eigenvalues, particularly $\lambda_1, \lambda_2$ and $\lambda_n$, over various 
graph families is an important and well-studied problem in spectral graph theory, see the 
surveys \cite{Cvetkovic_Rowlinson_1990, Cvetkovic_Simic_1995, Stanic_book}. Investigation of combinations 
of eigenvalues, such as Nordhaus-Gaddum type sums $\lambda_i(G) + \lambda_i(\overline{G})$ 
\cite{Aouchiche_Hansen_2013, Nikiforov_2007, Terpai_2011},  $\lambda_1 - \lambda_2$ 
(\emph{spectral gap}) \cite{Stanic_2013, Stanic_book}, $\lambda_1 - \lambda_n$ (\emph{spectral spread}) 
\cite{Breen_Riasanovsky_Tait_Urschel_2022}, and $\lambda_1 + \lambda_n$ (a measure of bipartiteness) 
\cite{Csikvari_2022}, has also led to interesting results, techniques and applications.

Recall that the energy $\mathcal{E}(G)$ of a graph $G$ on $n$ vertices is defined 
as the sum of the absolute values of eigenvalues of $G$. Equivalently,
\[
\mathcal{E}(G) = 2\max_{1\leq k\leq n} \sum_{i=1}^k \lambda_i(G).
\]
Thus, the sum of the $k$ largest eigenvalues of a graph sheds light on the energy of the graph. 
Mohar \cite{Mohar_2009} established an upper bound on $\sum_{i=1}^k \lambda_i(G)$ for 
$1 \leq k \leq n$ and provided further details on its role in theoretical chemistry. 
Mohar's bound was improved further by Das, Mojallal and Sun \cite{Das_Ahmad_Sun_2019}. 
Gernert\footnote{as described by Nikiforov \cite{Nikiforov_2006_2}.} investigated the case $k=2$ 
and proved that the inequality $\lambda_1 (G) + \lambda_2 (G) \leq n$ holds for various graph 
families, including regular graphs, triangle-free graphs and planar graphs. Gernert asked whether 
this inequality holds for any graph. Nikiforov \cite{Nikiforov_2006_2} answered this question 
in the negative and also proved that the spectral sum of graphs in $\mathcal{G}(n)$ is upper 
bounded by $2n/\sqrt{3}$. Later, Ebrahimi, Mohar, Nikiforov, and 
Ahmady \cite{Ebrahimi_Mohar_Nikiforov_Ahmady_2008} presented an improved upper bound on 
the spectral sum as described below.

\begin{theorem}[\cite{Ebrahimi_Mohar_Nikiforov_Ahmady_2008}]\label{thm:mohar_specsum}
For any graph $G\in \mathcal{G}(n)$, we have 
\[ 
\lambda_1(G)+\lambda_2(G) < \frac{8.0185}{7} n.
\] 
\end{theorem}

Moreover, in \cite{Ebrahimi_Mohar_Nikiforov_Ahmady_2008} the authors constructed an infinite 
family of graphs of order $n = 7k$ which have spectral sum $8k -2 = \frac{8}{7}n-2$ and 
suggested that $8n/7$ is the correct upper bound for the spectral sum. In this paper, we prove their conjecture.  

\begin{theorem}\label{thm:upper_bound_8_over_7_graph} 
For any $G\in \mathcal{G}(n)$, we have
\[
\lambda_1(G) + \lambda_2(G) \leq \frac{8}{7} n.
\]
\end{theorem}

We prove Theorem \ref{thm:upper_bound_8_over_7_graph} by translating the problem to graphons and 
then utilizing analytic tools to show that the spectral sum of graphons is at most $8/7$ 
(Theorem \ref{thm:upper_bound_8_over_7_graphon}). Our proof is inspired by Terpai's \cite{Terpai_2011} 
(cf. \cite{Liu_2024}) maximization proof of the sum $\lambda_1(G) + \lambda_1(\overline{G})$, 
and Breen, Riasanovsky, Tait and Urschel's \cite{Breen_Riasanovsky_Tait_Urschel_2022} maximization 
proof of the spread $\lambda_1 - \lambda_n$ using graphons. However, there is a significant 
difficulty when considering the spectral sum that we must overcome here which requires novel methods. 
When considering either $\lambda_1(G) + \lambda_1(\overline{G})$ or $\lambda_1 - \lambda_n$, 
both of these maximization problems can be expressed as the maximization of a two-term sum 
of quadratic forms over pairs of unit vectors (or $\mathscr{L}^2$-functions when working with graphons). 
Because of this, one can do alterations to the purported extremal graph (or graphon) 
and its eigenvectors and estimate from below the objective function by considering 
Rayleigh quotients. The major difficulty with the spectral sum is that in this case, 
one has to maximize a two-term sum of quadratic forms over pairs of \emph{orthogonal} 
vectors. The orthogonality condition usually breaks down when one tries to modify the eigenvector(s) 
on a subset of vertices, and so any modification needs to be chosen carefully. We introduce some new 
ideas to overcome this, and we believe that the techniques used in this paper may find applications 
in other problems where one wishes to optimize a sum of quadratic forms over orthogonal vectors.
\medskip

\noindent \textbf{High-level outline of our proof of Theorem \ref{thm:upper_bound_8_over_7_graph}.} We first 
show that it suffices to prove the weaker bound $\frac{8}{7}n + o(n)$ for the spectral sum of 
graphs (Theorem \ref{thm:upper_bound_8_over_7_asymptotic}) in Section \ref{sec:asymptotic_result}. 
We then give a brief introduction to graphons and relate the spectral sum of graphs with that of 
graphons (Lemma \ref{lemma:graph_to_graphon}) in Section \ref{sec:graphon_intro}. Thus, it suffices 
to show that the spectral sum of graphons is at most $8/7$ (Theorem \ref{thm:upper_bound_8_over_7_graphon}). 
We prove the result for graphons in Section \ref{sec:proof_graphon_theorem}. We begin with an 
extremal graphon $W(x,y)$ which has maximum spectral sum and maximum value of its largest eigenvalue 
among all such graphons. We establish the \emph{adjacency criteria} for $W(x,y)$ (Lemma \ref{lemma:adjacency}) which describes 
the points $(x,y)\in [0,1]^2$ where $W(x,y)=1$. Assuming $\mu_1$ and $f$ (resp. $\mu_2$ and $g$) 
denote the largest (resp. the second largest) eigenvalue and the corresponding eigenfunction of $W$, 
we show that $f$ and $g$ satisfy an ellipse equation for a.e. $x\in [0,1]$ (Lemma \ref{lemma:ellipse}). 
Using Carath\'{e}odory's Theorem from convex geometry, we then prove that $f$ and $g$ can be assumed 
to be step-functions with at most $6$ steps a.e. (Lemma \ref{lemma:six_steps_f_g}). We reduce 
the problem to analyzing the spectral sum of weighted graph(s) $G^*$ of order at most $6$ 
(Lemma \ref{lemma:weighted_graph}). Using structural arguments, we determine the possibilities 
for $G^*$ in subsection \ref{sec:structure_G_star} and show that it suffices to upper bound the 
spectral sum of the unique weighted graph $H_6$ shown in Figure \ref{fig:structure_G} (Lemma \ref{lemma:H_6}). 
We prove Lemma \ref{lemma:H_6} by first reducing it to a positive semi-definiteness verifiability 
problem (Lemma \ref{lemma:H_6_x}) using exterior algebra, and then using the matrix sum of squares 
technique from convex optimization to do the verification.  

\section{Asymptotic result suffices} 
\label{sec:asymptotic_result}

We observe the following using a graph-blowup argument originally by Nikiforov (cf. \cite{Csikvari_2009}). 

\begin{lemma}
Suppose for any graph $G\in \mathcal{G}(n)$, we have 
\[
\lambda_1(G) + \lambda_2(G) \le \frac{8n}{7} + o(n).    
\]
Then, Theorem \ref{thm:upper_bound_8_over_7_graph} holds for any $G\in \mathcal{G}(n)$.
\end{lemma}

\begin{proof}
For a graph $G\in \mathcal{G}(n)$, consider $G^{(t)}$, the graph obtained by replacing every 
vertex of $G$ by an empty graph of order $t$, where the copies of two vertices of $G$ are 
adjacent in $G^{(t)}$ if and only if the two vertices are adjacent in $G$. The graph $G^{(t)}$ 
is called the \emph{(open)}-$t$-\emph{blowup} of $G$. It is known (see \cite{Nikiforov_2006}) 
that the eigenvalues of $G^{(t)}$ are precisely $t\lambda_i(G)$ $(i = 1, \ldots, n)$ and $n(t-1)$ 
additional $0$'s. By assumption, we have 
\[ 
\lambda_1(G^{(t)}) + \lambda_2(G^{(t)}) \leq \frac{8tn}{7} + o(tn).
\]
Since $\lambda_1(G^{(t)}) = t\lambda_1(G)$ and $\lambda_2(G^{(t)}) = t\lambda_2(G)$, it follows that
\[ 
t\lambda_1(G) +  t\lambda_2(G) \leq \frac{8tn}{7} + o(tn).
\]
Dividing by $t$ and considering the limit as $t\rightarrow\infty$, we get
\[ 
\lambda_1(G) + \lambda_2(G) \leq \frac{8n}{7},
\] 
completing the proof of the lemma. 
\end{proof}

Thus, in order to prove Theorem \ref{thm:upper_bound_8_over_7_graph}, it is sufficient 
to prove the following asymptotic result. 

\begin{theorem}\label{thm:upper_bound_8_over_7_asymptotic} 
For any graph $G\in \mathcal{G}(n)$, we have 
\[
\lambda_1(G) + \lambda_2(G) \le \frac{8n}{7} + o(n).    
\]
\end{theorem}

\section{Enter graphons}
\label{sec:graphon_intro}

Graphons are analytic objects that generalize graphs, originally defined in an influential 
paper by Lov\'{a}sz and Szegedy \cite{Lovasz_Szegedy_2006} and the monograph by Lov\'{a}sz \cite{Lovasz_2012} (cf. \cite{Borgs_Chayes_Lovasz_Sos_Vesztergombi_2012}). 
We first give the necessary background on graphons, which is needed for further development in this paper.  

A \emph{graphon} is a symmetric Lebesgue measurable function $K:[0,1]^2 \rightarrow [0,1]$. 
We denote the set of graphons by $\mathcal{K}$. A graphon $K\in \mathcal{K}$ is 
called \emph{step-graphon} if there exists a partition $U_1, \ldots, U_k$ of $[0,1]$ 
such that $K$ is constant on $U_i\times U_j$ for all $i,j$. If $U\subseteq [0,1]^2$ 
is a symmetric (i.e., $(x,y)\in U$ if and only if $(y,x)\in U$) measurable set, 
then its indicator function $\chi_U$, defined by 
\[
\chi_U(x,y) = 
\begin{cases}
    1, & \text{if }(x,y) \in U; \\
    0, & \text{otherwise},
\end{cases}
\]
is clearly a step-graphon. A graphon generalizes the notion of a graph as follows: 
given a (unweighted) graph $G\in \mathcal{G}(n)$ with vertex set $V(G)=\{1, \ldots, n\}$, 
one can naturally define a subset
\[ 
\bigcup_{\substack{u\sim v\\ u,v\in V(G)}} \left[\frac{u-1}{n}, \frac{u}{n}\right]\times 
\left[\frac{v-1}{n}, \frac{v}{n}\right]\subseteq [0,1]^2,
\]
which is symmetric and measurable, whose indicator function (which we denote by $\chi_G$ for brevity) 
is a step-graphon. Clearly, for a given graph $G$, many graphons can be associated corresponding to 
different vertex labellings of $G$. Hence, an equivalence relation is defined on the set of graphons 
$\mathcal{K}$ using the \emph{cut norm} defined below so that for any graph $G$ there is a unique 
(up to isomorphism) equivalence class of graphons corresponding to $G$.
 
For $K\in \mathcal{K}$, its \emph{cut norm} $\| K \|_{\Box}$ is defined as 
\[ 
\| K \|_{\Box} = \sup_{S, T \subseteq [0,1]} \left| \int_{S\times T} K(x,y) \dif{x}\dif{y}\right|.
\]
Let $\delta_{\Box}$ be the \emph{cut semidistance} on $\mathcal{K}$ given by 
\[
\delta_{\Box}(U, W) = \inf\{\|U - W^{\varphi} \|_{\Box} : \varphi\ \text{is a measure-preserving bijection on}\ [0,1]\},
\]
where $U, W\in \mathcal{K}$, and $W^{\varphi}(x,y) = W(\varphi(x), \varphi(y))$. Define 
a relation $\sim$ on $\mathcal{K}$ as follows: 
$U \sim W$ if and only if $\delta_{\Box}(U, W) = 0$. 
It is known that $\sim$ is an equivalence relation and the quotient space $\mathcal{K}/\sim$ 
is a compact metric space, see \cite[Theorem 5.1]{Lovasz_Szegedy_2007}. 

For a measurable subset $U\subseteq [0,1]^2$, we will denote its Lebesgue measure by $m(U)$. 
We say that two measurable subsets of $[0,1]^2$ are equal \emph{almost everywhere} (a.e. for short) 
if their symmetric difference has measure zero. Consider the Hilbert space $\mathscr{L}^2[0,1]$ 
of square-integrable functions from $[0,1]$ to $\mathbb{R}$ with inner product given by 
$\langle f,g\rangle = \int_0^1 f(x)g(x)\dif{x}$. For $K\in \mathcal{K}/\sim$, define the 
Hilbert-Schmidt operator $A_K: \mathscr{L}^2[0,1]\rightarrow \mathscr{L}^2[0,1]$ by 
\[ 
(A_Kf)(x) = \int_0^1 K(x,y)f(y)\dif{y}
\]
for $f\in \mathscr{L}^2[0,1]$ and a.e. $x\in [0,1]$. As $K$ is symmetric and bounded, 
it is known that $A_K$ is a compact Hermitian operator (cf. \cite{Aubin_2000}). 
This means $A_K$ has a discrete, real spectrum with only possible accumulation point $0$. 
For two functions $f,g\in \mathscr{L}^2[0,1]$ and $K\in \mathcal{K}/\sim$, we define 
\[ 
fKg : = \int_0^1\!\int_0^1 K(x,y)f(x)g(y)\dif{x}\dif{y}.
\]

Recall the well-known Min-Max Theorem that relates eigenvalues to quadratic forms. 

\begin{theorem}[Min-Max Theorem]\label{thm:min_max} 
Let $A$ be a compact Hermitian operator on a Hilbert space $H$ with inner product 
$\langle\,\cdot\,, \,\cdot\,\rangle$. Let $U$ denote a subspace of $H$. 
Then the $k$-th largest eigenvalue $\lambda_k$ of $A$ is given by
\[ 
\lambda_k = \max_{\substack{U\\ \dim(U)=k}} ~ \min_{\substack{z\in U\\ \|z\|=1}} \langle z, A z\rangle.
\]
Moreover, if $z_1, \ldots, z_k\in H$ are mutually orthonormal vectors such that 
$\lambda_i = \langle z_i, A z_i\rangle$ for $i=1, \ldots, k$, then $z_i$ is a $\lambda_i$-eigenvector for $A$.  
\end{theorem}

Let $\mu_1(K)$ and $\mu_2(K)$ denote the largest and the second largest eigenvalue of $A_K$. 
Then, by Theorem \ref{thm:min_max}, we have   
\begin{equation}\label{eq:min-max-graphon}
 \mu_1(K) = \max_{\substack{f\in \mathscr{L}^2[0,1] \\ \| f\|_2=1}} fKf \quad \text{and}\quad
 \mu_2(K) = \max_{\substack{U\subset \mathscr{L}^2[0,1]\\ \dim(U)=2}} ~ \min_{\substack{g\in U \\ \| g\|_2=1}} gKg.
\end{equation}
Furthermore, if $f\in \mathscr{L}^2[0,1]$ is a unit function such that $\mu_1(K) = fKf$, 
then $f$ is a $\mu_1(K)$-eigenfunction of $K$ (up to a set of measure zero). Similarly, 
if $f$ is a $\mu_1(K)$-eigenfunction and $g\in \mathscr{L}^2[0,1]$ is a unit function 
orthogonal to $f$ such that $\mu_2(K) = gKg$, then $g$ is a $\mu_2(K)$-eigenfunction 
of $K$. Also, one can always choose a $\mu_1(K)$-eigenfunction $f$ which is non-negative. 

We define $\sigma(K): = \mu_1(K) + \mu_2(K)$ and we call it the \emph{spectral sum} 
of the graphon $K$. We note the following easy consequence of Theorem \ref{thm:min_max}, 
which we will use repeatedly.

\begin{lemma}\label{lemma:spectral_sum_graphon} 
For $K\in \mathcal{K}/\sim$, we have
\[
\sigma(K) = \mu_1(K) + \mu_2(K) = \max_{\substack{\| f\|_2=\| g\|_2 = 1\\ \langle f,g\rangle = 0}} fKf + gKg.
\]
\end{lemma}

The following result shows that $\mu_1$ and $\mu_2$ are continuous functions on 
the space $K/ \sim$ w.r.t. $\delta_\Box$. 

\begin{theorem}[cf. {\cite[Theorem 6.6]{Borgs_Chayes_Lovasz_Sos_Vesztergombi_2012}} or {\cite[Theorem 11.54]{Ore_1962}}]
Let $\{K_i\}_{i\in \mathbb{N}}$ be a sequence of graphons converging to $K$ with 
respect to $\delta_\Box$. Then as $i\rightarrow \infty$,  
\[
\mu_1(K_i) \rightarrow \mu_1(K)  \quad \text{and} \quad \mu_2(K_i) \rightarrow \mu_2(K).
\]
\end{theorem}

So if $K \sim \widetilde K$, then $\mu_1(K) = \mu_1(\widetilde{K})$ and $\mu_2(K) = \mu_2(\widetilde{K})$. 
By compactness of the quotient space $K \backslash \sim$, the following optimization problem is well defined: 
\[
\max_{K \in K \backslash \sim} \sigma(K).
\]
In particular, there exists a graphon (which we call \emph{extremal graphon}) that achieves the maximum. 

As expected, the spectral sum of a graph $G$ is related to the spectral sum of its 
corresponding graphon $\chi_G$ as described in the following lemma. It can be proved 
by replacing the eigenvectors of $G$ with their corresponding step-functions, 
see \cite{Liu_2024} or \cite{Terpai_2011} for details.

\begin{lemma}\label{lemma:graph_to_graphon} For a graph $G$ of order $n$, we have 
\[
\lambda_1(G) = n\mu_1(\chi_G)\quad \text{ and }\quad\lambda_2(G) = n\mu_2(\chi_G).
\]
\end{lemma}

In view of Lemma \ref{lemma:graph_to_graphon}, the following result implies 
Theorem \ref{thm:upper_bound_8_over_7_asymptotic}. 

\begin{theorem}\label{thm:upper_bound_8_over_7_graphon}
For any graphon $K\in \mathcal{K}/\sim$, we have $\sigma(K)\leq 8/7$. 
\end{theorem}

We prove Theorem \ref{thm:upper_bound_8_over_7_graphon} in the following section.

\section{Proof of Theorem \ref{thm:upper_bound_8_over_7_graphon}}
\label{sec:proof_graphon_theorem} 

Throughout this section, $W\in \mathcal{K}/\sim$ will denote a graphon that 
maximizes the spectral sum $\sigma(\,\cdot \,)$ over $\mathcal{K}/\sim$ and 
has the maximum value of its largest eigenvalue among all such graphons. 

We will denote by $\mu_1$ and $f$ (resp. $\mu_2$ and $g$) the largest (resp. the second largest) 
eigenvalue and the corresponding eigenfunction of $W$. This means 
\begin{equation}\label{eq:eigenvalue_equation_mu1_f}
(A_Wf)(x) = \int_0^1 W(x,y)f(y)\dif{y} = \mu_1f(x)
\end{equation}
and
\begin{equation}\label{eq:eigenvalue_equation_mu2_g}
(A_Wg)(x) = \int_0^1 W(x,y)g(y)\dif{y} = \mu_2g(x),     
\end{equation}
for a.e. $x\in [0,1]$. We further assume that $\|f\|_2^2 = \|g\|_2^2 = 1$ 
and $\langle f,g\rangle = 0$. Throughout, if the limits of integration of an integral 
are not specified, then it should be taken to be equal to $[0,1]^2$. 

In \cite{Ebrahimi_Mohar_Nikiforov_Ahmady_2008}, the authors constructed an infinite 
family of graphs of order $n$ with spectral sum $\frac{8}{7}n - 2$. In light of 
Lemma \ref{lemma:graph_to_graphon}, we see that 
\begin{equation}\label{eq:lower_bound}
 \mu_1 + \mu_2\ge \frac{8}{7}.  
\end{equation}
Since $\mu_1\le 1$, we conclude that $\mu_2 \geq 1/7$. In particular, this 
means that $\mu_2$ is always positive. We will often use this technical detail without mention in later proofs.

\subsection{Adjacency criterion}

Since $f$ is the eigenfunction for $\mu_1$, we can assume that $f$ is non-negative. 
Define the \emph{positive support} of $g$ to be the set $S^+(g) = \{x\in [0,1] :g(x) > 0\}$. 
The \emph{negative support} $S^-(g)$ and the \emph{zero set} $S^0(g)$ are defined analogously. 
The \emph{support} of $g$ is given by $S(g) = S^+(g)\cup S^-(g)$. For convenience, 
for any $(x,y)\in [0,1]^2$, we define 
\[
\kappa(x,y):= f(x)f(y) + g(x) g(y).
\] 
We first observe the following eigen-entry condition for an element $(x,y)\in [0,1]^2$ 
to be contained in the support of the extremal graphon $W$. When translated to graphs, 
this condition tells us when two vertices are adjacent in the extremal graph, and hence 
we call it the \emph{adjacency criterion}. 

\begin{lemma}[Adjacency]\label{lemma:adjacency} 
For a.e. $(x,y)\in [0,1]^2$, 
\[
W(x,y) = 
\begin{cases}
1, & \text{if}\ \kappa(x,y) \geq 0; \\
0, & \text{otherwise}.
\end{cases}
\]
\end{lemma}

\begin{proof}
Let $D_+ = \{(x,y)\in [0,1]^2: \kappa(x,y) > 0\}$, and define $D_0$ and $D_-$ similarly. 
Consider the graphon $W_1$ defined by 
\[
W_1(x,y) = 
\begin{cases}
1, & \text{if}\ (x,y)\in D_+\cup D_0;\\ 
0, & \text{if}\ (x,y)\in D_-.
\end{cases}
\]
By Lemma \ref{lemma:spectral_sum_graphon}, we have
\begin{align*}
\sigma (W_1) - \sigma(W)  
& \geq \int \big(W_1(x,y) - W(x,y)\big) \kappa(x,y) \dif{x}\dif{y} \\
& = \int_{D_+} \big(1 - W(x,y)\big) \kappa(x,y) \dif{x}\dif{y} - \int_{D_-} W(x,y)\kappa(x,y) \dif{x}\dif{y} \\
& \geq 0.
\end{align*}
Since $\sigma (W) \geq \sigma (W_1)$ by the choice of $W$, we have $\sigma(W_1) = \sigma(W)$, 
and both integrals in the last inequality must be $0$. This implies $W(x,y) = W_1(x,y)$ a.e. on $D_+\cup D_-$.

To address the behaviour on $D_0$, we consider the first eigenvalue $\mu_1(W_1)$. 
Applying \eqref{eq:min-max-graphon} and making use of the fact that $W(x,y) = W_1(x,y)$ a.e. on $D_+\cup D_-$ yield
\begin{align*}
  \mu_1(W_1) & \geq \int W_1(x,y)f(x)f(y)\dif{x}\dif{y} \\
  & = \int W(x,y)f(x)f(y)\dif{x}\dif{y} + \int_{D_0} \left(1 - W(x,y)\right)f(x)f(y)\dif{x}\dif{y} \\
  & = \mu_1 + \int_{D_0} \left(1 - W(x,y)\right)f(x)f(y)\dif{x}\dif{y}.
\end{align*}
Since $\mu_1\ge \mu_1(W_1)$ by the choice of $W$, we conclude $\mu_1(W_1) = \mu_1$ 
and the integral in the last inequality is 0. This implies $W(x,y)=1$ a.e. on $D_0$. The proof is complete.
\end{proof}

\subsection{Ellipse equation for $f$ and $g$}

In the next lemma, we show that for the graphon $W$, the function $\mu_1 f(x)^2 + \mu_2 g(x)^2$ 
is constant on the interval $[0,1]$ except maybe on a subset of measure zero.

\begin{lemma}[Ellipse equation]\label{lemma:ellipse}
For a.e. $x\in [0,1]$, we have    
\[
\mu_1 f(x)^2 + \mu_2 g(x)^2 = \mu_1 + \mu_2. 
\]
\end{lemma}

\begin{proof} 
Suppose the following is true: for any disjoint subintervals $I$ 
and $J$ of $[0,1]$ with $m(I)=m(J)$, we have 
\begin{equation}\label{eq:constant_on_intervals}
\int_I \mu_1 f(x)^2 + \mu_2 g(x)^2 \, \dif{x} = \int_J \mu_1 f(x)^2 + \mu_2\, g(x)^2 \,\dif{x}.
\end{equation}
Then $\mu_1 f(x)^2 + \mu_2\,g(x)^2$ is a constant function a.e.. Since 
\[
\int_0^1\mu_1 f(x)^2 + \mu_2\,g(x)^2 \dif{x} = \mu_1 +\mu_2,
\] 
the constant is $\mu_1 +\mu_2$ and the assertion holds. In what follows, 
we aim to establish equation \eqref{eq:constant_on_intervals}. 

Consider disjoint subintervals $I=[i_1, i_2]$ and $J=[j_1, j_2]$ of $[0,1]$ 
with length $m(I)=i_2-i_1=j_2- j_1=m(J)>0$. Without loss of generality, assume that
\[  
\int_I \mu_1 f(x)^2 + \mu_2\,g(x)^2 \, \dif{x} \leq \int_J \mu_1 f(x)^2 + \mu_2\, g(x)^2 \,\dif{x}.
\]
For any $\varepsilon > 0$ sufficiently small, consider the (unique) piecewise linear 
function $\varphi$ that stretches $I$ to length $(1 + \varepsilon) m(I)$, shrinks $J$ 
to length $(1 - \varepsilon) m(J)$, and shifts only the elements in between $I$ and $J$. 
Clearly, $\varphi: [0, 1] \to [0, 1]$ is an orientation-preserving homeomorphism such 
that $\varphi(0) =0$ and $\varphi(1)=1$. Furthermore, $\varphi$ is differentiable a.e. 
with derivative $\varphi'$ given by
\[
\varphi'(x) = 
\begin{cases}
1 + \varepsilon, & \text{if}\ x\in I; \\
1 - \varepsilon, & \text{if}\ x\in J; \\
1, & \text{otherwise}.
\end{cases}
\]

Let $U = \{(x,y)\in [0,1]^2: W(\varphi(x), \varphi(y)) > 0\}$. 
By Lemma \ref{lemma:adjacency}, we see that $U = \{(x,y)\in [0,1]^2: W(\varphi(x), \varphi(y)) = 1\}$.  
Consider the corresponding graphon $\chi_U$, and the functions 
\begin{equation}\label{eq:tilde-f-g-h}
\tilde{f}:= \varphi'\cdot (f \circ \varphi), ~~ 
\tilde{g}:= \varphi'\cdot (g \circ \varphi) ~~\text{and}~~ \tilde{h}:=\tilde{g} - \alpha \tilde{f},
\end{equation}
where $\alpha := \frac{\langle \tilde{f}, \tilde{g}\rangle}{\langle \tilde{f}, \tilde{f}\rangle}$. 
Clearly, $\langle \tilde{f}, \tilde{h}\rangle = 0$. 

Now, we compute the $\mathscr{L}^2$-norms of $\tilde{f}$, $\tilde{g}$, and $\tilde{h}$. 
Applying the change of variable $u = \varphi(x)$ and using \eqref{eq:tilde-f-g-h}, we obtain
\begin{align}\label{eq:f_tilde_norm}
\| \tilde{f} \|_2^2 
& = \int_{I} (1+\varepsilon) \varphi'(x)\, f(\varphi(x))^2 \dif{x} 
+ \int_{J} (1-\varepsilon) \varphi'(x)\, f(\varphi(x))^2 \dif{x} 
+ \int_{[0,1]\backslash (I\cup J)} \varphi'(x)\, f(\varphi(x))^2 \dif{x} \nonumber \\
& = \int_{\varphi(I)} (1+\varepsilon)\, f(u)^2 \dif{u} 
+ \int_{\varphi(J)} (1-\varepsilon)\, f(u)^2 \dif{u} 
+ \int_{[0,1]\backslash (\varphi(I)\cup \varphi(J))} f(u)^2 \dif{u}\nonumber \\
& = \int_{[0,1]} f(u)^2 \dif{u} + \varepsilon \left(\int_{\varphi(I)} f(u)^2 \dif{u} 
- \int_{\varphi(J)} f(u)^2 \dif{u} \right) \nonumber \\
& = 1 + \varepsilon \left( \int_{I} f(u)^2 \dif{u} - \int_{J} f(u)^2 \dif{u} \right) \nonumber \\
& \phantom{=} + \varepsilon \left(\int_{\varphi(I)\setminus I} f(u)^2 \dif{u} 
- \int_{I\setminus\varphi(I)} f(u)^2 \dif{u} + \int_{J\setminus \varphi(J)} f(u)^2 \dif{u} 
- \int_{\varphi(J)\setminus J} f(u)^2 \dif{u} \right)\nonumber \\
& = 1 + \varepsilon \left(\| \chi_I f \|_2^2 - \| \chi_J f \|_2^2\right) + o(\varepsilon).
\end{align}
Similar to the calculation above, we have
\begin{equation}\label{eq:g_tilde_norm}
\|\tilde{g}\|_2^2 = 1 + \varepsilon \left(\| \chi_I g \|_2^2 - \|\chi_J g\|_2^2\right) + o(\varepsilon).
\end{equation}
On the other hand, by \eqref{eq:tilde-f-g-h} we see
\[
\|\tilde{g}\|_2^2 = \langle \tilde{g}, \tilde{g}\rangle 
= \langle \tilde{h} + \alpha\tilde{f}, \tilde{h} + \alpha \tilde{f}\rangle 
= \langle \tilde{h}, \tilde{h}\rangle + \alpha^2 \langle \tilde{f}, \tilde{f}\rangle, 
\]
since $\langle \tilde{f}, \tilde{h}\rangle = 0$. It follows that 
\begin{equation}\label{eq:h_tilde_norm}
\|\tilde{h}\|_2^2 = \|\tilde{g}\|_2^2 - \alpha^2\|\tilde{f}\|_2^2 \leq \|\tilde{g}\|_2^2.
\end{equation}

Next, we shall compute the quadratic forms $\tilde{f}\chi_U \tilde{f}$, 
$\tilde{g}\chi_U\tilde{g}$ and $\tilde{h}\chi_U\tilde{h}$. 
Using the substitutions $u = \varphi (x)$ and $v = \varphi (y)$, we see that
\begin{align}\label{eq:quadratic_form_tilde_f}
\tilde{f} \chi_U \tilde{f} 
& = \int \chi_{U}(x,y)\ \tilde{f}(x) \tilde{f}(y) \dif{x}\dif{y} \nonumber \\
& = \int \chi_{U}(x,y)\ \varphi'(x)f(\varphi(x)) \ \varphi'(y)f(\varphi(y)) \dif{x}\dif{y} \nonumber \\
& = \int W(u,v) f(u)f(v) \dif{u}\dif{v} = \mu_1.
\end{align}
Similarly, $\tilde{g} \chi_U \tilde{g} = \mu_2$, and
\begin{align*} \tilde{g} \chi_U \tilde{f} 
& = \tilde{f} \chi_U \tilde{g} \\ 
& = \int \chi_{U}(x,y) \tilde{f}(x) \tilde{g}(y) \dif{x}\dif{y} \\
& = \int W(\varphi(x), \varphi(y)) \varphi'(x)f(\varphi(x)) \varphi'(y)g(\varphi(y)) \dif{x}\dif{y} \\
& = \int W(u,v) f(u)g(v) \dif{u}\dif{v} \\
& = \int_{[0,1]} \left(\int_{[0,1]} W(u,v) g(v)\dif{v}\right) f(u)\dif{u} \\
& = \mu_2 \int_{[0,1]} g(u)f(u) \dif{u} = 0.
\end{align*}
Combining the above equations with \eqref{eq:tilde-f-g-h}, we deduce that
\begin{align}\label{eq:quadratic_form_tilde_h}
\tilde{h} \chi_U \tilde{h} 
& = \big(\tilde{g}(x) - \alpha \tilde{f}(x)\big) \chi_U \big(\tilde{g}(x) - \alpha \tilde{f}(x)\big) \nonumber \\
& = \tilde{g}(x)\chi_U\tilde{g}(x) + \alpha^2\tilde{f}(x)\chi_U\tilde{f}(x) - 2\alpha\tilde{f}(x)\chi_U\tilde{g}(x) \nonumber \\
& = \mu_2 + \alpha^2 \mu_1 \geq \mu_2.
\end{align}
Finally, we compare $\sigma(W)$ and $\sigma(\chi_U)$. By the choice of $W$, 
Lemma \ref{lemma:spectral_sum_graphon} and \eqref{eq:f_tilde_norm} -- \eqref{eq:quadratic_form_tilde_h}, we find
\begin{align*}
0 & \leq \sigma(W) - \sigma(\chi_U) \\
& \leq (\mu_1 + \mu_2) - \left(\frac{\tilde{f}\chi_U \tilde{f}}{\|\tilde{f} \|_2^2} 
+ \frac{\tilde{h} \chi_U \tilde{h}}{\| \tilde{h} \|_2^2}\right) 
\leq (\mu_1 + \mu_2) - \left(\frac{\tilde{f}\chi_U \tilde{f}}{\| \tilde{f} \|_2^2 } 
+ \frac{\tilde{h} \chi_U \tilde{h}}{\| \tilde{g} \|_2^2 }\right) \\
& \leq (\mu_1 + \mu_2) - \left(\frac{\mu_1}{1 + \varepsilon \left(\|\chi_I f \|_2^2 - \|\chi_J f \|_2^2\right) + o(\varepsilon)} 
+ \frac{\mu_2}{1 + \varepsilon \left(\| \chi_I g \|_2^2 - \| \chi_J g \|_2^2\right) + o(\varepsilon)}\right) \\
& = \mu_1 \varepsilon \left(\| \chi_I f \|_2^2 - \| \chi_J f \|_2^2\right) 
+ \mu_2 \varepsilon \left(\| \chi_I g \|_2^2 - \| \chi_J g \|_2^2\right) + o(\varepsilon) \\
& = \varepsilon \left(\int_I \mu_1 f(x)^2 + \mu_2\,g(x)^2 \ \dif{x} - \int_J \mu_1 f(x)^2 + \mu_2\, g(x)^2 \ \dif{x}\right) + o(\varepsilon),
\end{align*}
as $\varepsilon\to 0$. If $\int_I \mu_1 f(x)^2 + \mu_2\,g(x)^2 \, \dif{x} < \int_J \mu_1 f(x)^2 + \mu_2\, g(x)^2 \,\dif{x}$, 
then choosing $\varepsilon$ sufficiently small makes the right-hand side of the above inequality negative, 
which is a contradiction. We conclude that equation \eqref{eq:constant_on_intervals} holds and the proof is complete. 
\end{proof}

\subsection{$f$ and $g$ are step functions with at most $6$ steps}

Throughout, let
\[
E := \{x\in [0,1]: x\ \text{satisfies \eqref{eq:eigenvalue_equation_mu1_f}, \eqref{eq:eigenvalue_equation_mu2_g} 
and the ellipse equation in Lemma \ref{lemma:ellipse}} \}.
\]
By Lemma \ref{lemma:ellipse}, we see that $E$ is dense in $[0,1]$. So in what follows, 
we will work with $E$ (or $E\times E$) so that we can assume the eigenvalue-equation 
and the ellipse equation without worry. Let 
\[
E^+: =E\cap S^+(g), \quad E^0:=E\cap S^0(g)\quad \text{and}\quad E^-:=E\cap S^-(g).
\] 

We first prove a technical lemma about the eigenfunctions $f$ and $g$. 

\begin{lemma}\label{lemma:mixed-cloning}
Let $U\subseteq S^+(g)\cup S^0(g)$ be a measurable subset with positive measure. 
Suppose there exist points $x_1,\ldots,x_k\in U\cap E$ and non-negative real numbers 
$u_1, \ldots, u_k$ with $\sum_i u_i = m(U)$ such that the following two moment constraints are satisfied: 
\begin{equation}\label{eq:moment-constraints}
\sum_{i=1}^k u_i f(x_i)^2= \int_U f(x)^2 \dif{x},
\qquad
\sum_{i=1}^k u_i f(x_i)g(x_i)= \int_U f(x)g(x) \dif{x}, 
\end{equation}
Then the linear moments are also preserved; that is,
\[
\sum_{i=1}^k u_i f(x_i)= \int_U f(x) \dif{x},
\qquad
\sum_{i=1}^k u_i g(x_i)= \int_U g(x) \dif{x}.
\] 
By symmetry, a similar assertion holds if $U\subseteq S^-(g)\cup S^0(g)$.
\end{lemma}

\begin{proof} 
Consider a partition of $U$ into $k$ measurable subsets $U_1, \ldots, U_k$ 
such that $m(U_i) = u_i$ for all $1\le i\le k$. Define a new graphon $\widetilde{W}$ with 
respect to the points $x_i$ and the subsets $U_i$ such that
\[
\widetilde{W}(x,y)=
\begin{cases}
W(x,y), & x,y\notin U;\\
W(x_i,y), & x\in U_i,\ y\notin U; \\
1, & x,y\in U.
\end{cases}
\]
Essentially, we are \emph{cloning} the point $x_i$ over $U_i$. Moreover, since $W(x,y)=1$ 
for all $x,y\in U$, the graphons $W$ and $\widetilde{W}$ agree on $U\times U$. 
Define functions $\tilde{f},\tilde{g}$ using a similar cloning operation, i.e.,  
\begin{equation}\label{eq:tilde-def}
\tilde f(x)=
\begin{cases}
f(x_i), & x\in U_i;\\
f(x), & x\notin U;
\end{cases}
\qquad
\tilde g(x)=
\begin{cases}
g(x_i), & x\in U_i;\\
g(x), & x\notin U.
\end{cases}
\end{equation}

We have the following two claims.

\begin{claim}\label{claim:norm_tilde_f_g}
$\|\tilde{f}\|_2 = \|\tilde{g}\|_2 = 1$, and $\langle \tilde f,\tilde g\rangle = 0$. 
\end{claim}

\begin{proof}[Proof of Claim \ref{claim:norm_tilde_f_g}]
Using \eqref{eq:moment-constraints} we get 
\[
\|\tilde f\|_2^2
= \int_{U^c} f(x)^2 \dif{x} + \sum_{i=1}^k u_i f(x_i)^2
= \left(1-\int_U f(x)^2 \dif{x}\right) + \sum_{i=1}^k u_i f(x_i)^2 = 1.
\]
Here, $U^c = [0,1]\backslash U$. Again, by Lemma \ref{lemma:ellipse}, $g(x_i)^2$ 
is a function of $f(x_i)^2$ given by
\[
g(x_i)^2=\frac{\mu_1+\mu_2-\mu_1 f(x_i)^2}{\mu_2}.
\]
Averaging with weights $u_i$ and using the first constraint in \eqref{eq:moment-constraints} gives
\begin{equation}\label{eq:sum_aver_g2_xi}
\sum_{i=1}^k u_i g(x_i)^2
=\frac{(\mu_1+\mu_2)\, m(U) - \mu_1 \sum u_i f(x_i)^2}{\mu_2}.
\end{equation}
On the other hand, integrating the ellipse equation from Lemma \ref{lemma:ellipse} over $U$ yields
\[
\mu_1 \int_U f(x)^2 \dif{x} + \mu_2\int_U g(x)^2 \dif{x} = (\mu_1+\mu_2)\, m(U),
\] 
which is equivalent to
\begin{equation}\label{eq:integration_g2_over_U}
\int_U g(x)^2 \dif{x} = \frac{(\mu_1+\mu_2)\, m(U) - \mu_1\int_U f(x)^2 \dif{x}}{\mu_2}.
\end{equation}
Combining \eqref{eq:sum_aver_g2_xi} and \eqref{eq:integration_g2_over_U} we see 
$\sum_i u_i g(x_i)^2=\int_U g(x)^2 \dif{x}$. Thus
\[
\|\tilde g\|_2^2
= \int_{U^c} g(x)^2 \dif{x} + \sum_i u_i g(x_i)^2
= 1-\int_U g(x)^2 \dif{x} + \sum_{i=1}^k u_i g(x_i)^2
= 1.
\]

Finally, since $\langle f,g\rangle=0$, we have 
\[
\int_{U^c} f(x)g(x) \dif{x} = -\int_U f(x)g(x) \dif{x},
\]
and hence the second constraint in \eqref{eq:moment-constraints} gives us
\[
\langle \tilde f,\tilde g\rangle
= \int_{U^c} f(x)g(x) \dif{x} + \sum_{i=1}^k u_i f(x_i)g(x_i)=0,
\]
finishing the proof of the claim.
\end{proof}

\begin{claim}\label{claim:tilde_fWf}
We have
\begin{equation}\label{eq:mu_1_W_tilde}
\tilde f\ \widetilde{W}\tilde f = \mu_1+\left(\int_U f(x) \dif{x} -\sum_{i=1}^k u_i f(x_i)\right)^2,
\end{equation}
and 
\begin{equation}\label{eq:mu_2_W_tilde}
\tilde g\ \widetilde{W}\tilde g 
=\mu_2+\left(\int_U g(x) \dif x -\sum_{i=1}^k u_i g(x_i)\right)^2.
\end{equation}
\end{claim}

\begin{proof}[Proof of Claim \ref{claim:tilde_fWf}]
We only prove \eqref{eq:mu_1_W_tilde}; the proof of \eqref{eq:mu_2_W_tilde} is similar. 
Expand $\tilde f\ \widetilde{W}\tilde f$ into the following blocks
\begin{align*}
\tilde f\ \widetilde{W}\tilde f
& =\iint_{U^c\times U^c} W(x,y)f(x)f(y)\,\dif{x} \dif{y} \\
& \quad + 2\sum_{i=1}^k \int_{U_i}\int_{U^c} W(x_i,y)\, f(x_i)\, f(y) \dif{x} \dif{y} \\
& \quad + \iint_{U\times U} \tilde f(x)\tilde f(y) \dif{x} \dif{y}.
\end{align*}
Note that the second line equals $2\sum_i u_i f(x_i)\int_{U^c} W(x_i,y)f(y)\,\dif y$, while the third line can be written as
$\big(\int_U \tilde f(x) \dif x \big)^2 = \big(\sum_i u_i f(x_i)\big)^2$ by \eqref{eq:tilde-def}. So
\begin{equation}\label{eq:Qmix-f}
\tilde f\ \widetilde{W}\tilde f
= I+2\sum_{i=1}^k u_i f(x_i)\int_{U^c} W(x_i,y)f(y)\,\dif{y} + \left(\sum_{i=1}^k u_i f(x_i)\right)^2,
\end{equation}
where $I:=\iint_{U^c\times U^c} W(x,y) f(x) f(y) \dif{x} \dif{y}$. Using the 
eigenvalue-equation at $x\in U\cap E$, we have 
\begin{equation}\label{eq:row-outside}
\int_{U^c} W(x,y)f(y)\,\dif{y} = \mu_1 f(x)-\int_U f(y) \dif{y}.
\end{equation}
Using \eqref{eq:row-outside} at $x=x_i$ we rewrite the mixed term in \eqref{eq:Qmix-f}:
\[
2\sum_{i=1}^k u_i f(x_i) \int_{U^c}W(x_i,y)f(y)\,\dif{y} 
= 2\mu_1\sum_{i=1}^k u_i f(x_i)^2 - 2\sum_{i=1}^k u_i f(x_i)\int_U f(y) \dif{y}. 
\]
Combining the above equation with \eqref{eq:Qmix-f} yields
\begin{equation}\label{eq:Qmix-f-2}
\tilde f\, \widetilde{W}\tilde f
= I + 2\mu_1\sum_{i=1}^k u_i f(x_i)^2-2\sum_{i=1}^k u_i f(x_i)\int_U f(x)\dif{x}+\left(\sum_{i=1}^k u_i f(x_i)\right)^2.
\end{equation}

Now, we compute the term $\mu_1=f\, Wf$ using the same block expansion. 
Applying \eqref{eq:row-outside} for $x\in U\cap E$, we obtain
\begin{align*}
\mu_1 
& = I + 2\int_U f(x)\left(\mu_1 f(x)-\int_U f(y)\dif{y}\right) \dif{x} + \left(\int_U f(x)\dif{x}\right)^2 \\
& = I + 2\mu_1\int_U f(x)^2 \dif{x} - \left(\int_U f(x) \dif{x}\right)^2.
\end{align*}
Subtract this identity from \eqref{eq:Qmix-f-2} to get
\[
\tilde f\ \widetilde{W}\tilde f -\mu_1
=2\mu_1\left(\sum_{i=1}^k u_i f(x_i)^2 - \int_U f(x)^2 \dif{x} \right)
+\left(\int_U f(x)\dif{x} -\sum_{i=1}^k u_i f(x_i)\right)^2.
\]
By \eqref{eq:moment-constraints}, the first term vanishes, and we get the desired claim. 
\end{proof}

Because $\tilde f,\tilde g$ are orthonormal, using Lemma \ref{lemma:spectral_sum_graphon}, we have
\begin{align*}
\sigma(\widetilde{W})
& \geq \tilde f\ \widetilde{W}\tilde f + \tilde g\ \widetilde{W}\tilde g \\
& = \mu_1+\mu_2 + \left(\int_U f(x) \dif{x} -\sum_i u_i f(x_i)\right)^2 
+ \left(\int_U g(x) \dif{x} -\sum_i u_i g(x_i)\right)^2.
\end{align*}
Since $W$ is extremal, $\sigma(W)\ge \sigma(\widetilde{W})$, and so the integrals in the 
last inequality must be zero. The assertion follows.
\end{proof}

Note that in the above lemma, the assumption $U\subseteq S^+(g)\cap S^0(g)$ allows 
us to specify the value of $\widetilde{W}$ on $U$, i.e., $\widetilde{W}(x,y) = 1 = W(x,y)$ 
for a.e. $x,y\in U$. This is necessary for the calculation to work. 

Next, we recall two results from the theory of convexity. 

\begin{lemma}[{\cite[Section 2.1.4]{Boyd_Vandenberghe_2004}}]\label{lemma:expectation_convex_hull}
Let $C \subseteq \mathbb{R}^d$ be a convex set and $\bm{x}$ be a random vector such that 
$\bm{x} \in C$ with probability one. Then $\mathbb{E}(\bm{x}) \in C$.
\end{lemma}

For a subset $S \subseteq \mathbb{R}^d$, the \emph{convex hull} of $S$, denoted by $\mathrm{conv}(S)$, 
is defined to be the smallest convex set that contains it. We recall the following well-known theorem 
by Carath\'eodory on the convex hull of sets. 

\begin{theorem}[Carath\'eodory's Theorem {\cite[Theorem 1.2.3]{Matousek_2002}}]\label{theorem:Caratheodory}
Let $S \subseteq \mathbb{R}^d$. If $\bm{x}\in\mathrm{conv}(S)$, then there exist points
$s_0,s_1,\ldots,s_d \in S$ and coefficients $a_0,a_1,\ldots,a_d \geq 0$ with
$\sum_{i=0}^d a_i = 1$ such that
\[
\bm{x} = \sum_{i=0}^d a_i s_i.
\]
Equivalently, every point in $\mathrm{conv}(S)$ is a convex combination of at most $d+1$ points of $S$.
\end{theorem}

We are now ready to show that $f,g$ are step-functions with at most 6 steps. 

\begin{lemma}\label{lemma:six_steps_f_g}
We can assume that $W$ has the following property: except on a set of measure zero, 
$f$ \textup{(}and therefore $g$\textup{)} takes at most $3$ values on $S^+(g)\cup S^0(g)$ 
and at most $3$ values on $S^0(g)\cup S^-(g)$. 
\end{lemma}

\begin{proof} 
Let $U:=S^+(g)\cup S^0(g)$ and define
\[
A:=\frac{1}{m(U)} \int_U f(x)^2\,\dif{x},\quad 
B:=\frac{1}{m(U)} \int_U f(x)g(x)\,\dif{x},\quad
C:=\frac{1}{m(U)} \int_U f(x)\,\dif{x}.
\]
Recall that $E^+\cup E^0\subseteq U$ is such that $m(E^+\cup E^0)=m(U)$. 
Define the map $\Phi:E^+\cup E^0\to\mathbb R^2$ by
\[
\Phi(x):=\bigl(f(x)^2,\ f(x)g(x)\bigr).
\]
Then, by the definition of $A$ and $B$,
\[
(A,B)=\frac{1}{m(U)} \int_U \Phi(x)\,\dif{x} = \frac{1}{m(U)} \int_{E^+\cup E^0} \Phi(x)\,\dif{x}.
\]
Using Lemma \ref{lemma:expectation_convex_hull}, we see that $(A,B)\in \mathrm{conv}(\Phi(E^+\cup E^0))$. 
By Carath\'eodory's Theorem (Theorem \ref{theorem:Caratheodory}), there exist points $x_1,x_2,x_3\in E^+\cup E^0$ 
and weights $t_1,t_2,t_3\ge 0$ with $t_1+t_2+t_3=1$ such that
\[
t_1\Phi(x_1)+t_2\Phi(x_2)+t_3\Phi(x_3)=(A,B),
\]
which is equivalent to 
\begin{equation}\label{eq:momentmatch}
\sum_{i=1}^3 t_i f(x_i)^2 = A,\qquad
\sum_{i=1}^3 t_i f(x_i)g(x_i) = B.
\end{equation}

Next, we show that necessarily
\begin{equation}\label{eq:meanmatch}
t_1 f(x_1)+t_2 f(x_2)+t_3 f(x_3)=C.
\end{equation}
To this end, set $u_i:=t_i\, m(U)$ and partition $U$ into subsets $U_1, U_2, U_3$ 
such that $m(U_i) = u_i$ for all $i$. Then, \eqref{eq:momentmatch} becomes
\[
\int_U f(x)^2\,\dif{x} = \sum_{i=1}^3 u_i f(x_i)^2,\qquad
\int_U f(x)g(x)\,\dif{x} = \sum_{i=1}^3 u_i f(x_i)g(x_i).
\]
Using Lemma \ref{lemma:mixed-cloning} we get $\sum_{i=1}^3 u_i f(x_i)=\int_U f(x)\,\dif{x}$, 
which is equivalent to \eqref{eq:meanmatch}.

Finally, define $\widetilde{W}$, $\tilde{f}$ and $\tilde{g}$ as in the proof of 
Lemma \ref{lemma:mixed-cloning} w.r.t. $U$, $x_i$'s, and $U_i$'s described above. Then, 
the proof of Lemma \ref{lemma:mixed-cloning} gives us that $\sigma(\widetilde{W})=\sigma(W)$. 
Moreover, \eqref{eq:mu_1_W_tilde} implies that $\mu_1(\widetilde{W})\ge \mu_1$. By the choice of $W$, 
we see that $\mu_1(\widetilde{W}) = \mu_1$. This means $\widetilde{W}$ is an extremal graphon 
with maximum largest eigenvalue. Furthermore, by \eqref{eq:mu_1_W_tilde}, we have 
$\mu_1(\widetilde{W})=\tilde{f}\, \widetilde{W}\tilde{f}$, implying $\tilde{f}$ is an 
$\mu_1(\widetilde{W})$-eigenfunction. Also, it is clear that 
$\mu_2 = \mu_2(\widetilde{W})=\tilde g \widetilde{W}\tilde{g}$ using \eqref{eq:mu_2_W_tilde}, 
which implies $\tilde{g}$ is an $\mu_2(\widetilde{W})$-eigenfunction. Thus, $\widetilde{W}$ 
is an extremal graphon with the property that its eigenfunctions $\tilde{f}$ and $\tilde{g}$ take at most 3 values on $U$. 
Now, starting with $\widetilde{W}$ and $U=S^0(g)\cup S^-(g)$ and repeating the above steps, 
we can get an extremal graphon with the desired property. This completes the proof. 
\end{proof}

\subsection{Monotonicity of $f$ and $g$}

Without loss of generality, we can assume that $g$ is monotonically decreasing on $[0,1]$. 
By the assumption on $g$, we see that $f$ is monotonically increasing on $E^+$, 
constant on $E^0$ and monotonically decreasing on $E^-$. In light of Lemma \ref{lemma:adjacency}, 
it is clear that $W(x,y)=1$ whenever $x,y\in E^+\cup E^0$ or $x,y\in E^-\cup E^0$. 
The value of $W(x,y)$ is undetermined at this stage whenever $x\in E^+$ 
and $y\in E^-$. We do have the following monotonicity property.

\begin{lemma}[Monotonicity of neighbourhoods]\label{lemma:monotonicity} 
The following statements hold.
\begin{enumerate}[$(i)$]
\item Let $x_1,x_2\in E^+$ and $y\in E^-$ such that $x_2\geq x_1$. If $W(x_1,y)=1$, then $W(x_2,y)=1$.   
\item Let $x_1,x_2\in E^-$ and $y\in E^+$ such that $x_2\leq x_1$. If $W(x_2,y)=1$, then $W(x_1,y)=1$. 
\end{enumerate}
\end{lemma}

\begin{proof}
Since $x_2\geq x_1\in E^+$, we have $f(x_2)\geq f(x_1)$ and $g(x_2)\leq g(x_1)$. 
Since $g(y)<0$, we see that 
\[
f(x_2)f(y)+g(x_2)g(y)\geq f(x_1)f(y) + g(x_1)g(y).
\]
Assertion $(i)$ now follows from Lemma \ref{lemma:adjacency} and the assumption that $W(x_1, y)=1$. 
A similar argument works for assertion $(ii)$.
\end{proof}

\subsection{Back to (weighted) graphs}

In Lemma \ref{lemma:six_steps_f_g}, we have shown that eigenfunctions $f,g$ are step-functions 
with at most $6$ steps a.e. on the interval $[0,1]$. So consider a partition 
$U_1\sqcup \cdots \sqcup U_k$ of $[0,1]$ ($k\leq 6$) such that for $i=1,\ldots, k$, we have 
\begin{itemize}
\item $g\equiv \beta_i$ on $U_i$ a.e. for some numbers $\beta_1>\cdots >\beta_k$, and
\item $f\equiv \alpha_i$ on $U_i$ a.e., where $\alpha_i\ge 0$ is related to $\beta_i$ 
by the ellipse equation in Lemma \ref{lemma:ellipse}.
\end{itemize}

Define $u_i: = m(U_i)$. Clearly, $\sum_{i=1}^k u_i = 1$. Define a graph $G^*$ (possibly with loops) 
such that $V(G) = \{U_1, \ldots, U_k\}$ and $U_i\sim U_j$ whenever $W(x,y) = 1$ on $U_i\times U_j$ 
a.e. Using Lemma \ref{lemma:adjacency} it is clear that $U_i\sim U_i$ for all $i$, i.e., every 
vertex in $G^*$ has a loop. Now, consider the weighted matrix 
\[
M^* := D_{u}^{1/2} A(G^*)D_{u}^{1/2} = 
\begin{cases}
    \sqrt{u_iu_j}, & \text{if }U_i\sim U_j; \\
    0, & \text{otherwise}.
\end{cases}
\]
Here $A(G^*)$ is the adjacency matrix of $G^*$ with rows and columns indexed by 
$U_1, \ldots, U_k$ and $D_{u}$ denotes the diagonal matrix with entries $\{u_1, \ldots, u_k\}$. 
We relate the spectral sum of the extremal graphon $W$ to the spectral sum of the matrix $M^*$.

\begin{lemma}\label{lemma:weighted_graph}
We have $\mu_1 = \lambda_1(M^*)$ and $\mu_2 = \lambda_2(M^*)$.
\end{lemma}

\begin{proof}
By the eigenvalue-equation \eqref{eq:eigenvalue_equation_mu1_f} for $\mu_1$, 
for any $i = 1, \ldots, k$, we have 
\begin{align*}
\mu_1\alpha_i = \sum_{U_j\sim U_i} \alpha_j u_j 
\end{align*}
which is equivalent to 
\begin{align*}
\mu_1 \sqrt{u_i}\alpha_i = \sum_{U_j\sim U_i} \sqrt{u_j} \alpha_j \sqrt{u_ju_i}.
\end{align*}
This implies $\mu_1$ is an eigenvalue of $M^*$ with eigenvector $(\sqrt{u_i}\,\alpha_i)_{i=1}^k$. 

Conversely, if $(a_1, \ldots, a_k)$ is an eigenvector for $\lambda_1(M^*)$, 
then it is clear that the step-function given by 
\[
h(x) = \frac{a_i}{\sqrt{u_i}}, \quad x\in U_i
\]
is an eigenfunction of $W$ with eigenvalue $\lambda_1(M^*)$. We conclude that $\mu_1 = \lambda_1(M^*)$. 

Similarly, one can argue that $\mu_2 = \lambda_2(M^*)$.
\end{proof}

Note that the lengths $u_i$ sum to $1$ and are undetermined at this stage. 
The spectral sum of the matrix $M^*$ depends on these lengths, equivalently on 
the diagonal matrix $D_u$, even if the underlying graph $G^*$ is the same. Let us define 
\[
\sigma(G^*) := \max_{D_u} \sigma(M^*),
\]
and call it the \emph{spectral sum} of $G^*$. 

Thus, to finish the proof of Theorem \ref{thm:upper_bound_8_over_7_graphon}, we only 
need to show that the spectral sum $\sigma(G^*)$ is at most $8/7$. To that end, 
we next determine the possible graphs $G^*$ for each $1\leq k\leq 6$. 

\subsection{Structure of $G^*$}
\label{sec:structure_G_star}

Precisely, we will show the following.

\begin{lemma}\label{lemma:structure_of_G} 
The graph $G^* \in \{P_3, P_4, H_5, H_6\}$, where $P_3, P_4, H_5, H_6$ 
are as shown in Figure \ref{fig:structure_G}. 
\end{lemma}

\begin{figure}[H]
\begin{subfigure}{0.49\textwidth}
\centering
    \begin{tikzpicture}[
    vertex/.style={circle,draw,minimum size=6mm,inner sep=0pt, thick, fill = gray!30}
]

\node[vertex] (v1) at (0,0) {$1$};
\node[vertex] (v2) at (2,0) {$2$};
\node[vertex] (v3) at (4,0) {$3$};

\draw (v1) -- (v2) -- (v3);

\path[-,min distance=1cm] (v1)edge[in=65,out=115,above] node {}(v1);
\path[-,min distance=1cm] (v2)edge[in=65,out=115,above] node {}(v2);
\path[-,min distance=1cm] (v3)edge[in=65,out=115,above] node {}(v3);
\end{tikzpicture}
    \caption{$P_3$}
\end{subfigure}
\hfill
\begin{subfigure}{0.49\textwidth}
    \centering
 \begin{tikzpicture}[
    vertex/.style={circle,draw,minimum size=6mm,inner sep=0pt, thick, fill = gray!30}
]

\node[vertex] (v1) at (0,0) {$1$};
\node[vertex] (v2) at (2,0) {$2$};
\node[vertex] (v3) at (4,0) {$3$};
\node[vertex] (v4) at (6,0) {$4$};

\draw (v1) -- (v2) -- (v3) -- (v4);

\path[-,min distance=1cm] (v1)edge[in=65,out=115,above] node {}(v1);
\path[-,min distance=1cm] (v2)edge[in=65,out=115,above] node {}(v2);
\path[-,min distance=1cm] (v3)edge[in=65,out=115,above] node {}(v3);
\path[-,min distance=1cm] (v4)edge[in=65,out=115,above] node {}(v4);
\end{tikzpicture}
    \caption{$P_4$}
\end{subfigure}

\begin{subfigure}{0.49\textwidth}
    \centering
\begin{tikzpicture}[
    vertex/.style={circle,draw,minimum size=6mm,inner sep=0pt, thick, fill = gray!30}
]

\node[vertex] (v1) at (0,0) {$1$};
\node[vertex] (v2) at (2,0) {$2$};
\node[vertex] (v3) at (3,2) {$3$};
\node[vertex] (v4) at (4,0) {$4$};
\node[vertex] (v5) at (6,0) {$5$};

\draw (v1) -- (v2) -- (v4) -- (v5);
\draw 
(v1) -- (v3)
(v2) -- (v3)
(v4) -- (v3)
(v5) -- (v3);

\path[-,min distance=1cm] (v1)edge[in=65,out=115,above] node {}(v1);
\path[-,min distance=1cm] (v2)edge[in=65,out=115,above] node {}(v2);
\path[-,min distance=1cm] (v3)edge[in=65,out=115,above] node {}(v3);
\path[-,min distance=1cm] (v4)edge[in=65,out=115,above] node {}(v4);
\path[-,min distance=1cm] (v5)edge[in=65,out=115,above] node {}(v5);
\end{tikzpicture}    
    \caption{$H_5$}
\end{subfigure}
\hfill
\begin{subfigure}{0.49\textwidth}
    \centering
 \begin{tikzpicture}[
    vertex/.style={circle,draw,minimum size=6mm,inner sep=0pt, thick, fill = gray!30}
]

\node[vertex] (v1) at (0,0) {$1$};
\node[vertex] (v2) at (2,0) {$2$};
\node[vertex] (v3) at (3,2) {$3$};
\node[vertex] (v4) at (4,0) {$4$};
\node[vertex] (v5) at (6,0) {$5$};
\node[vertex] (v6) at (5,-2) {$6$};

\draw (v1) -- (v2) -- (v4) -- (v5);
\draw 
(v1) -- (v3)
(v2) -- (v3)
(v4) -- (v3)
(v5) -- (v3)
(v4) -- (v6)
(v5) -- (v6);

\path[-,min distance=1cm] (v1)edge[in=65,out=115,above] node {}(v1);
\path[-,min distance=1cm] (v2)edge[in=65,out=115,above] node {}(v2);
\path[-,min distance=1cm] (v3)edge[in=65,out=115,above] node {}(v3);
\path[-,min distance=1cm] (v4)edge[in=65,out=115,above] node {}(v4);
\path[-,min distance=1cm] (v5)edge[in=65,out=115,above] node {}(v5);
\path[-,min distance=1cm] (v6)edge[in=65,out=115,above] node {}(v6);
\end{tikzpicture}   
    \caption{$H_6$}
\end{subfigure}
    \caption{Possibilities for $G^*$. The vertex labeled by $i$ in the figure corresponds to $U_i$.}
    \label{fig:structure_G}
\end{figure}
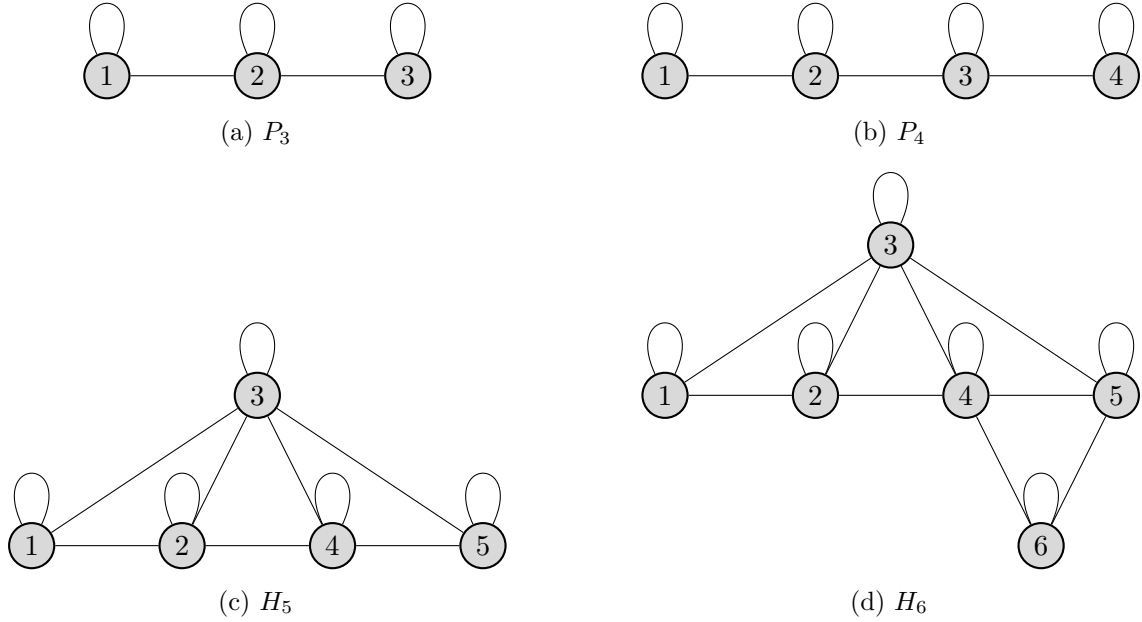

We will prove Lemma \ref{lemma:structure_of_G} in a series of claims. We say that two 
distinct vertices in a graph are \emph{true twins} if they have the same closed neighbourhoods. 
We observe that $G^*$ cannot have true twins.

\begin{lemma}\label{lemma:true_twin_free} 
The graph $G^*$ is true-twin-free.
\end{lemma}

\begin{proof}
Suppose $U_i, U_j\in V(G^*)$ are true twins. Then, by the eigenvalue-equation \eqref{eq:eigenvalue_equation_mu1_f} 
for $\mu_1$, we have 
\[
\mu_1 \alpha_i = u_i \alpha_i + u_j\alpha_j + \sum_{\ell \neq i,j}u_\ell \alpha_\ell = \mu_1 \alpha_j, 
\]
implying $\alpha_i = \alpha_j$, a contradiction.
\end{proof}

Let us now analyze the structure of $G^*$ for different values of $k$.

\begin{claim}
If $k\le 3$, then $G^*\cong P_3$.
\end{claim}

\begin{proof}
The case $k=1$ is not possible. Indeed, $1 = \| f\|_2 = \alpha_1$ and $1 = \| g\|_2 = |\beta_1|$. 
This gives $0 = \langle f,g\rangle = \alpha_1\beta_1 \neq 0$, a contradiction. 

If $k=2$, then $\sigma(M^*) = \tr(M^*) = \sum_{i=1}^2 u_i = 1$ for all choices of $D_u$. 
By \eqref{eq:lower_bound}, we know that $\sigma(G^*)\geq 8/7$. Thus, there is no graph $G^*$ of order $2$. 

Assume now that $k=3$. Without loss of generality, we can assume that $\beta_1 > \beta_2 \geq 0 > \beta_3$, 
since $g$ and $-g$ are both eigenfunctions for $\mu_2$. Then, by Lemma \ref{lemma:adjacency}, 
we see that $U_1\sim U_2$ in $G^*$. It is easy to check that the only graph on $3$ vertices 
with at least one edge and true-twin-free is $P_3$ (obviously with a loop on each vertex). 
\end{proof}

We already know that $U_i\sim U_i$ for all $1\le i\le k$, i.e., each vertex of $G^*$ has loops. 
In the proofs of the claims below, we will repeatedly use the fact that $G^*$ is true-twin free 
(Lemma \ref{lemma:true_twin_free}) and monotonicity of neighbourhoods (Lemma \ref{lemma:monotonicity}) without mention.

\begin{claim} 
Suppose $k = 4$, then $G^* \cong P_4$.    
\end{claim}

\begin{proof} In light of Lemma \ref{lemma:six_steps_f_g}, one of the following occurs:

\textbf{Case 1:} $\beta_1 > \beta_2 > \beta_3\ge 0 > \beta_4$.

Then, by Lemma \ref{lemma:adjacency}, $U_1, U_2, U_3$ induce a triangle in $G^*$. If $U_4\sim U_2$, 
then $U_3\sim U_4$, implying $U_2$ and $U_3$ are true-twins, a contradiction. If $U_4\nsim U_2$, 
then $U_4\nsim U_1$, and so $U_1$ and $U_2$ are true-twins, again a contradiction.

\textbf{Case 2:} $\beta_1 > \beta_2 > 0 > \beta_3 > \beta_4$.

We have $U_1\sim U_2$ and $U_3\sim U_4$. If $U_1\sim U_4$, then $U_1, U_2$ are adjacent to 
both $U_3, U_4$, implying $U_1$ and $U_2$ are true-twins, a contradiction. So suppose $U_1\nsim U_4$. 

If $U_1\sim U_3$, then $U_2\sim U_3$. If $U_2\nsim U_4$, then again $U_1$ and $U_2$ are true-twins, 
a contradiction. Thus, $U_2\sim U_4$. But then $U_2$ and $U_3$ are true-twins, a contradiction. 
So $U_1\nsim U_3, U_4$. Arguing similarly with $U_4$, we see that $U_4\nsim U_1, U_2$. 

Now, if $U_2\nsim U_3$, then $U_1$ and $U_2$ are true-twins, a contradiction. Thus, the only 
possibility is $U_2\sim U_3$. This completely determines $G^*$, and we have $G^*\cong P_4$. 
\end{proof}

\begin{claim} 
If $k = 5$, then $G^*\cong H_5$.    
\end{claim}

\begin{proof}
Using Lemma \ref{lemma:six_steps_f_g}, without loss of generality, we can assume that 
\[ 
\beta_1 > \beta_2 > \beta_3 \ge 0 > \beta_4 > \beta_5.
\] 
Thus, $U_1,U_2, U_3$ induce a triangle in $G^*$. Moreover, $U_4\sim U_5$. We consider the following cases:

\textbf{Case 1:} $U_1\sim U_5$. 

Then $U_1$ is adjacent to all vertices in $G^*$. This means $U_2$ is also adjacent to all vertices, 
implying $U_1$ and $U_2$ are true-twins, a contradiction.

\textbf{Case 2:} $U_1\nsim U_5$ and $U_1\sim U_4$.

Then $U_1, U_2, U_3$ are all adjacent to $U_4$. If $U_2\nsim U_5$, then $U_1$ and $U_2$ 
are true-twins, a contradiction. So suppose $U_2\sim U_5$. This implies $U_2$ and $U_3$ 
are adjacent to all vertices in $G^*$ and hence they are true-twins, a contradiction.

\textbf{Case 3:} $U_1\nsim U_4, U_5$. 

If $U_2\sim U_5$, then $U_2$ and $U_3$ are adjacent to all vertices, and hence they are true-twins, 
a contradiction. So suppose $U_2\nsim U_5$. If $U_2\nsim U_4$, then $U_1$ and $U_2$ are true-twins, 
a contradiction. So assume that $U_2\sim U_4$. We have $N[U_2] = \{U_1, U_2, U_3, U_4\}$. 
If $U_3\nsim U_5$, then $U_2$ and $U_3$ are true-twins, a contradiction. So we have $U_3\sim U_5$, 
i.e., $N[U_3] = \{U_1, U_2, U_3, U_4, U_5\}$. This completely determines $G^*$ and we see that $G^*\cong H_5$. 
\end{proof}

\begin{claim} 
If $k = 6$, then $G^*\cong H_6$.    
\end{claim}

\begin{proof} 
By Lemma \ref{lemma:six_steps_f_g}, we have 
\[ 
\beta_1 > \beta_2 > \beta_3 > 0 > \beta_4 > \beta_5 > \beta_6.
\] 
By Lemma \ref{lemma:adjacency} we see that $U_1, U_2, U_3$ (resp. $U_4, U_5, U_6$) induce 
a triangle in $G^*$. We consider the following cases:

\textbf{Case 1:} $U_1 \sim U_6$.
    
Then $N[U_1] = \{U_1, \ldots, U_6\}$. This implies $U_2, U_3$ are also adjacent to all vertices in $G^*$. 
But then $U_1$ and $U_2$ are true-twins, a contradiction.

\textbf{Case 2:} $U_1\nsim U_6$ and $U_1\sim U_5$.

Then $U_1\sim U_4$ and thus $U_2, U_3$ are also adjacent to $U_4, U_5$. If $U_2\nsim U_6$, 
then $U_1$ and $U_2$ are true-twins. If $U_2\sim U_6$, then $U_3\sim U_6$ and so $U_2$ and 
$U_3$ are true-twins. In either case, we have a contradiction.

\textbf{Case 3:} $U_1\nsim U_5, U_6$ and $U_1\sim U_4$. 

Then $U_2, U_3$ are adjacent to $U_4$. 
    
\textbf{Subcase 3.1:} $U_2\sim U_6$. Then $U_2$ and $U_3$ are adjacent to all vertices in $G^*$, a contradiction. 

\textbf{Subcase 3.2:} $U_2\nsim U_6$ and $U_2\sim U_5$. ~  Then $U_2,U_3$ are adjacent to $U_4, U_5$. 
If $U_3\nsim U_6$, then $U_2$ and $U_3$ are true-twins. If $U_3\sim U_6$, then $U_3$ and $U_4$ are true-twins. 
In both cases, we have a contradiction.

\textbf{Subcase 3.3:} $U_2\nsim U_5, U_6$ and $U_2\sim U_4$. ~ Then $U_1$ and $U_2$ are true-twins, a contradiction.

\textbf{Case 4:} $U_1\nsim U_4, U_5, U_6$.

Arguing as above, one can show that $U_6$ also is not adjacent to $U_1,U_2, U_3$. So we only need to 
worry about the edges (or non-edges) between $U_2, U_3$ and $U_4, U_5$. If $U_2\sim U_5$, then $U_2, U_3$ 
are adjacent to $U_4, U_5$ and so $U_2, U_3$ are true-twins, a contradiction. So assume $U_2\nsim U_5$. 
If $U_2\nsim U_4$, then $U_1$ and $U_2$ are true-twins, a contradiction. So we have $N[U_2] = \{U_1, U_2, U_3, U_4\}$. 
If $U_3\nsim U_5$, then $U_2$ and $U_3$ are true-twins, a contradiction. So the only possibility is 
$N[U_3] = \{U_1, U_2, U_3, U_4, U_5\}$. This determines the graph $G^*$. Indeed, $G^*\cong H_6$. 
\end{proof}

The above claims complete the proof of Lemma \ref{lemma:structure_of_G}. Interestingly, 
the graphs $P_3, P_4, H_5$ are induced subgraphs of $H_6$. This means that if $k < 6$, 
we can add sub-intervals with measure zero to the list $U_1, U_2, \ldots, U_k$ so that 
we have a partition of $[0,1]$ into $6$ sub-intervals. Moreover, we can define $f$ and $g$ 
on these additional measure zero sets so that $g$ always has non-negative values on $U_1, U_2, U_3$ 
and negative values on $U_4, U_5, U_6$ a.e. In other words, $g\equiv \beta_i$ on $U_i$ a.e. where 
\[
\beta_1 > \beta_2 > \beta_3 \ge 0 > \beta_4 > \beta_5 > \beta_6.
\] 
Thus, the problem is now reduced to proving the following.

\begin{lemma}\label{lemma:H_6} 
For $1\le i\le 6$, let $0\le u_i\le 1$ be such that $\sum_i u_i = 1$. Let $D_u$ denote the diagonal 
matrix with diagonal entries $u_1, \ldots, u_6$. Let 
\[
M^* = D_{u}^{1/2} A(H_6)D_{u}^{1/2}.
\]
Then $\sigma(M^*)\le 8/7$. 
\end{lemma}

We will prove the above Lemma \ref{lemma:H_6} using matrix sum of squares technique from 
convex optimization. Lemma \ref{lemma:H_6} is a non-convex optimization problem which we 
first convert to a higher dimensional convex optimization problem using exterior algebra. 

\subsection{Exterior algebra}

Let $M$ be a real square matrix of order $n$. For matrices $M,N$, let $M \otimes N$ denote 
the usual Kronecker product of matrices. Define 
\[ 
N:=M\otimes I_n + I_n\otimes M.
\]  
A combinatorial interpretation of $N$ is that if $M$ is the adjacency matrix of a graph, 
then $N$ is the adjacency matrix of the cartesian product of the graph with itself. 

For two vectors $\bm{u}, \bm{v}\in \mathbb{R}^n$, the \emph{wedge product} of $\bm{u}$ and $\bm{v}$ is given by 
\[ 
\bm{u}\wedge \bm{v}:= \bm{u}\otimes \bm{v} - \bm{v}\otimes \bm{u}. 
\]
Let $e_1, \ldots, e_n$ denote the standard basis of $\mathbb{R}^n$. Consider 
the $2$-nd \emph{exterior power} of $\mathbb{R}^n$ given by
\[ 
\Lambda^2 (\mathbb{R}^n)=\operatorname{span}\{e_i\wedge e_j: 1\leq i<j\leq n\},
\]
which is an antisymmetric subspace of the tensor space $\mathbb{R}^n\otimes \mathbb{R}^n$. 
An orthonormal basis of $\Lambda^2 (\mathbb{R}^n)$ is given by
\[
\mathcal{B} := \left\{\frac{e_i \wedge e_j}{\sqrt{2}}: 1\leq i<j\leq n\right\},
\]
implying that the dimension of $\Lambda^2 (\mathbb{R}^n)$ is $\binom{n}{2}$. 
We know that $\bm{u}\wedge \bm{v}\in \Lambda^2 (\mathbb{R}^n)$ for all $\bm{u}, \bm{v}\in \mathbb{R}^n$. 
See, for instance, \cite{MacLane_Birkhoff_1988} for details.  

We observe that the subspace $\Lambda^2 (\mathbb{R}^n)$ is invariant under the linear transformation $N$.

\begin{proposition}\label{prop:N_invariant} 
For all $\bm{u}, \bm{v}\in \mathbb{R}^n$, we have 
\[ 
N(\bm{u}\wedge \bm{v}) = (M\bm{u})\wedge \bm{v} + \bm{u}\wedge (M\bm{v}).
\]
\end{proposition}

\begin{proof}
Recall that $(A \otimes B)(C \otimes D) = AC \otimes BD$ for any square matrices $A, B, C, D$ of a given order. Then 
\begin{align*}
N(\bm{u}\wedge \bm{v}) &= \left(M \otimes I_n + I_n \otimes M \right) (\bm{u} \otimes \bm{v} - \bm{v} \otimes \bm{u}) \\
& = (M \otimes I_n)(\bm{u} \otimes \bm{v}) + (I_n \otimes M)(\bm{u} \otimes \bm{v}) - (M \otimes I_n)(\bm{v} \otimes \bm{u}) - (I_n \otimes M)(\bm{v} \otimes \bm{u}) \\
& = (M\bm{u} \otimes \bm{v} - \bm{v} \otimes M\bm{u}) + (\bm{u} \otimes M\bm{v} - M\bm{v} \otimes \bm{u}) \\
& = (M\bm{u})\wedge \bm{v} + \bm{u}\wedge (M\bm{v}). \qedhere
\end{align*}
\end{proof}

Let $P \in \mathbb{R}^{n^2 \times \binom{n}{2}}$ be the matrix whose columns are the basis 
vectors in $\mathcal{B}$. Consider the map 
$\psi: \mathbb{R}^{n\times n}\rightarrow \mathbb{R}^{\binom{n}{2} \times \binom{n}{2}}$ given by
\begin{equation}\label{eq:psi-M}
\psi(M):= P^{\mathrm{T}} N P = P^{\mathrm{T}}(M\otimes I_n+I_n\otimes M)P
\end{equation}
for any $M\in \mathbb{R}^{n\times n}$.

\begin{proposition}
The eigenvalues of $\psi(M)$ are precisely
\[ 
\lambda_i(M)+\lambda_j(M), \qquad 1\leq i<j\leq n. 
\]
\end{proposition}

\begin{proof}
Suppose $\bm{v}_1,\dots,\bm{v}_n$ denote the orthonormal eigenvectors (represented in the standard basis 
of $\mathbb{R}^n$) of $M$  corresponding to the eigenvalues $\lambda_1(M)\ge \cdots \ge \lambda_n(M)$, 
respectively. For $i<j$, denote by $(\bm{v}_i\wedge \bm{v}_j)_{\mathcal{B}}$ the representation of $\bm{v}_i\wedge \bm{v}_j$ 
in terms of the basis $\mathcal{B}$ as an element of $\Lambda^2 (\mathbb{R}^n)$. Using 
Proposition \ref{prop:N_invariant} and the fact that $P^{\mathrm{T}} P = I_{\binom{n}{2}}$, we see that
\begin{align*}
\psi(M)(\bm{v}_i\wedge \bm{v}_j)_{\mathcal{B}} 
& = (P^{\mathrm{T}} NP)(\bm{v}_i\wedge \bm{v}_j)_{\mathcal{B}} \\
& = P^{\mathrm{T}} N (\bm{v}_i\wedge \bm{v}_j) \\
& = P^{\mathrm{T}} \big((M\bm{v}_i)\wedge \bm{v}_j + \bm{v}_i\wedge (M\bm{v}_j)\big) \\
& = \big(\lambda_i(M)+\lambda_j(M)\big) P^{\mathrm{T}} (\bm{v}_i\wedge \bm{v}_j) \\
& = \big(\lambda_i(M)+\lambda_j(M)\big) (\bm{v}_i\wedge \bm{v}_j)_{\mathcal{B}}.
\end{align*}
Thus, $(\bm{v}_i\wedge \bm{v}_j)_{\mathcal{B}}$ is an eigenvector of $\psi(M)$ corresponding 
to the eigenvalue $\lambda_i(M)+\lambda_j(M)$. 
\end{proof}

\begin{remark}
It is worth mentioning that $\psi(M)$ defined in \eqref{eq:psi-M} is exactly the 
second additive compound matrix of $M=[m_{ij}]$. More generally, let $1 \leq k \leq n$, 
and let $\alpha, \beta\in \binom{[n]}{k}$. If $|\alpha\cap \beta|= k-1$, denote 
$\mathrm{sign} (\alpha,\beta) = (-1)^{|\{r\in\alpha\cap\beta,\ i<r<j\}|}$, 
where $i<j$ are the two unique elements in the symmetric difference $\alpha\Delta\beta$.
The \emph{$k$-th additive compound} of $M$ is the matrix
$M^{[k]}\in\mathbb{R}^{\binom{n}{k}\times\binom{n}{k}}$ defined by
\[
(M^{[k]})_{\alpha,\beta} =
\begin{cases}
\sum_{i\in\alpha} m_{ii}, & \text{if}\ \alpha = \beta, \\
\mathrm{sign} (\alpha, \beta)\cdot m_{ij}, & \text{if}\ |\alpha\cap\beta| = k-1, \alpha\backslash\beta = \{i\}, \beta\backslash\alpha = \{j\}, \\
0, & \text{otherwise}.
\end{cases}
\]
For further results on additive compounds, we refer the reader to \cite{Fiedler_1974}.
\end{remark}

Letting $\bm{x} = (x_1, \ldots, x_6)$, where $x_i = \sqrt{u_i}$ $(i =1, \ldots, 6)$, 
it is clear that the following implies Lemma \ref{lemma:H_6}.

\begin{lemma}\label{lemma:H_6_x} 
Let $\bm{x}\in \mathbb{R}^6$ be such that $\Vert \bm{x} \Vert_2 = 1$. 
Define $M^*(\bm{x}) = \diag(\bm{x}) A(H_6) \diag(\bm{x})$. Then
\[  
\frac87 I_{15}-\psi(M^*(\bm{x}))\succeq 0,
\]
i.e., $\frac87 I_{15}-\psi(M^*(\bm{x}))$ is a positive semi-definite matrix.
\end{lemma}

We prove Lemma \ref{lemma:H_6_x} using the sum of squares technique. 
We refer the reader to \cite{Parrilo_Thomas_2020} for further details.

\subsection{Matrix sum of squares}

Suppose we can find a symmetric positive semi-definite matrix $Q \in \mathbb{R}^{105 \times 105}$ 
and a symmetric matrix $T \in \mathbb{R}^{15 \times 15}$ such that
\begin{equation}\label{eq:Q_PSD}
\frac87 I_{15}-\psi(M^*(\bm{x})) = V(\bm{x})^{\mathrm{T}} QV(\bm{x})+(1-\|\bm{x}\|_2^2)T 
\end{equation}
for all $x\in \mathbb{R}^6$. Here 
\[ 
V(\bm{x}):= (1,x_1,\dots,x_6)^T\otimes I_{15}. 
\]
Then on the unit sphere
$\|\bm{x}\|_2=1$ the second term vanishes, and we obtain
\[ 
\frac87 I_{15}-\psi(M^*(\bm{x})) = V(\bm{x})^{\mathrm{T}} QV(\bm{x})\succeq 0, 
\]
implying Lemma \ref{lemma:H_6_x}. 

Exact rational matrices $Q$ and $T$, satisfying \ref{eq:Q_PSD} were found and are given \href{https://github.com/Shivaramkratos/Spectral_sum_codes/blob/main/verify_sos_identity.py}{here}. 
For details on how this was done, refer Appendix. This completes the (long!) proof of 
Theorem \ref{thm:upper_bound_8_over_7_graphon}.

\section{Concluding remarks}

In this paper, we established a tight upper bound for the spectral sum of graphs in $\mathcal{G}(n)$. 
The problem of characterizing graphs which attain the bound remains open 
(see Conjecture \ref{conj:spectral_sum_extremal_graphs} below).

For $n\geq p\geq q\geq 0$, define $K(n,p,q)$ to be the graph obtained by taking the complement of 
the union of the complete bipartite graph $K_{p,q}$ and $n-p-q$ isolated vertices. In other words, 
$K(n,p,q)$ is the join of $K_{n-p-q}$ with $K_p \cup K_q$. In their survey of automated conjectures 
in spectral graph theory, Aouchiche and Hansen \cite{Aouchiche_Hansen_2010} generalized the 
tight construction for spectral sum from \cite{Ebrahimi_Mohar_Nikiforov_Ahmady_2008} and framed the following conjecture. 

\begin{conjecture}[\cite{Aouchiche_Hansen_2010,Ebrahimi_Mohar_Nikiforov_Ahmady_2008}]\label{conj:spectral_sum_extremal_graphs}
Let $n\geq 5$. Then for all graphs $G\in \mathcal{G}(n)$, we have
\[
\lambda_1(G)+\lambda_2(G)\leq \lambda_1(K(n,p,q))+\lambda_2(K(n,p,q)),
\]
where if $n=7k+r$ with $0\leq r\le 6$, then 
\[
(p,q)= 
\begin{cases}
(2k,2k) & \text{if }r=0,1;\\
(2k+1,2k) & \text{if }r=2;\\
(2k+1,2k+1) & \text{if }r=3, 4;\\
(2k+2,2k+1) &\text{if }r=5;\\
(2k+2,2k+2) &\text{if }r=6.
\end{cases} 
\] 
Moreover, equality holds if and only if $G\cong K(n,p,q)$. 
\end{conjecture}

The key reason our proof of Theorem \ref{thm:upper_bound_8_over_7_graph} does not fully characterize 
the extremal graphs (equivalently, extremal graphons) is because of the extra assumption on $W$ that 
it has the maximum largest eigenvalue $\mu_1$ among all graphons that maximize the spectral sum. 
We require this extra assumption on $W$ at two places: first, in Lemma \ref{lemma:adjacency} to 
decide the value of $W(x,y)$ when $\kappa(x,y) = 0$; secondly, in Lemma \ref{lemma:six_steps_f_g} 
so that we can argue that $\tilde{f}$ is an eigenfunction for $\widetilde{W}$ corresponding to 
eigenvalue $\mu_1(\widetilde{W}) = \mu_1$, which finally allows us to choose $\widetilde{W}$ 
in place of $W$. We believe that bypassing this assumption on $W$ should lead to a characterization of the extremal graphon.

The problem of maximizing/minimizing the spectral sum for trees was recently resolved by 
Kumar, Mohar, Pragada and Zhan \cite{Kumar_Mohar_Pragada_Zhan_2026}. Indeed, they considered 
the convex combination $\alpha \lambda_1 + (1-\alpha)\lambda_2$ $(\alpha \in [0,1])$ of $\lambda_1$ 
and $\lambda_2$ of trees of order $n$ and fully characterized the extremal trees that maximize the convex combination for large $n$. 
It would be of interest to consider the convex combination of $\lambda_1$ and $\lambda_2$ for 
graphs in $\mathcal{G}(n)$ and find the optimal upper bound. One may also try to investigate the 
spectral sum of other families; the family of $K_r$-free graphs is of some interest. 
The minimization of spectral sum of graphs remains an open problem. We propose the 
following conjecture for interested readers.

\begin{conjecture} 
For sufficiently large $n$, the path uniquely minimizes the spectral 
sum among all connected graphs of order $n$.
\end{conjecture}

\section*{Acknowledgement}
Lele Liu is supported by the National Nature Science Foundation of China (No.\ 12471320), 
and Anhui Provincial Natural Science Foundation for Excellent Young Scholars (No.\ 2408085Y003). 
Hermie Monterde is supported in part by the Pacific Institute for the Mathematical Sciences through 
the PIMS-Simons Postdoctoral Fellowship. Michael Tait is partially supported by the US National 
Science Foundation via the grant DMS-2245556. We acknowledge the use of GPT 5.2 Plus during the ideation phase. 

\bibliographystyle{plain}
\bibliography{references_new}

\vspace{0.4cm}

\affl{Hitesh Kumar}{hitesh.kumar.math@gmail.com, hitesh\_kumar@sfu.ca}{Dept.\ of Mathematics, Simon Fraser University, Burnaby, BC \ V5A 1S6, Canada}

\affl{Lele Liu}{liu@ahu.edu.cn}{School of Mathematical Sciences, Anhui University, Hefei 230601, 
P.R. China}

\affl{Hermie Monterde}{hermie.monterde@uregina.ca}{Dept.\ of Mathematics and Statistics, University of Regina, Regina, SK, Canada S4S 0A2}

\affl{Shivaramakrishna Pragada}{shivaramakrishna\_pragada@sfu.ca}{Dept.\ of Mathematics, Simon Fraser University, Burnaby, BC \ V5A 1S6, Canada}

\affl{Michael Tait}{michael.tait@villanova.edu}{Dept.\ of Mathematics \& Statistics, Villanova University}

\newpage

\section*{Appendix}
\label{sec:Appendix}
Here, we describe the method we used to find the PSD matrix $Q$ and matrix $T$. For $1\leq i, j\leq 6$, 
let $E_{ij}$ denote the standard basis vector of $\mathbb{R}^{6\times 6}$ with $(i,j)$-th entry $1$ and 
$0$ elsewhere. Note that the map $\psi$ is linear. Then
\begin{equation}
   \psi(M(\bm{x})) = \sum_{i=1}^6 x_i^2 \psi(E_{ii}) +\sum_{1\leq i<j\leq 6} x_i x_j \psi(E_{ij} + E_{ji})A(H_{6})_{ij}. 
\end{equation}
Write $Q$ in $7\times 7$ blocks as $Q=(Q_{ab})_{0\leq a,b\leq 6}$, $Q_{ab}\in \mathbb R^{15\times 15}$. 
Expanding $V(\bm{x})^{\mathrm{T}} QV(\bm{x})$ gives
\begin{equation}
V(\bm{x})^{\mathrm{T}} QV(\bm{x}) =
Q_{00}
+\sum_{i=1}^6 x_i(Q_{0i}+Q_{i0})
+\sum_{i=1}^6 x_i^2 Q_{ii}
+\sum_{1\leq i<j\leq 6} x_i x_j(Q_{ij}+Q_{ji}).
\end{equation}
Also,
\begin{equation}
(1-\|x\|_2^2)T = T-\sum_{i=1}^6 x_i^2 T.
\end{equation}
Comparing coefficients using \eqref{eq:Q_PSD} we obtain the system
\begin{align}
    Q_{00}+T & =\frac87 I_{15}\nonumber \\
    Q_{0i}+Q_{i0} & =0 \nonumber \\
    Q_{ii}-T & =-\psi(E_{ii}) \nonumber \\
    Q_{ij}+Q_{ji} & =-\psi(E_{ij} + E_{ji})A(H_{6})_{ij}.
\end{align}
Eliminating $T$, this is equivalent to
\begin{align} \label{Eqn:Coeff_matching}
Q_{00}+Q_{ii} & =\frac87 I_{15}-\psi(E_{ii})\nonumber \\
Q_{0i}+Q_{i0} & =0 \nonumber \\
Q_{ij}+Q_{ji} & =-\psi(E_{ij} + E_{ji})A(H_{6})_{ij}.
\end{align}

We find $Q$ and $T$ and do the necessary verification as follows:
\begin{enumerate}
    \item Set up the coefficient equations \eqref{Eqn:Coeff_matching} as a semidefinite feasibility 
    problem in the unknown blocks $Q_{ab}$ and $T$, together with the constraint $Q\succeq 0$.
    \item Solve this feasibility problem numerically in \texttt{CVXPY} using the solver \texttt{SCS}, 
    to a prescribed error tolerance, obtaining numerical matrices $Q^{\mathrm{num}}$ and $T^{\mathrm{num}}$.
    \item Symmetrize the numerical solution by replacing
    \[
    Q^{\mathrm{num}}\leftarrow \frac12\bigl(Q^{\mathrm{num}}+(Q^{\mathrm{num}})^{\mathrm{T}}\bigr),
    \qquad
    T^{\mathrm{num}}\leftarrow \frac12\bigl(T^{\mathrm{num}}+(T^{\mathrm{num}})^{\mathrm{T}}\bigr).
    \]
    \item Convert the numerical matrices into exact rational matrices using \texttt{SymPy}. 
    More precisely, selected entries of $Q^{\mathrm{num}}$ and $T^{\mathrm{num}}$ are replaced by nearby rational numbers of bounded denominator, and then assembled into candidate exact matrices $Q^{\mathrm{rat}}$ and $T^{\mathrm{rat}}$.
    \item Reconstruct the remaining entries of $Q^{\mathrm{rat}}$ exactly from the linear coefficient equations, 
    so that the polynomial identity holds exactly over $\mathbb{Q}$.
    \item Verify the coefficient identity exactly using \texttt{SymPy} with rational arithmetic.
    \item Verify $Q^{\mathrm{rat}}\succeq 0$ exactly by checking a rank one decomposition 
    of $Q^{\mathrm{rat}}$ over $\mathbb{Q}$, again using \texttt{SymPy}.
\end{enumerate}
\end{document}